\documentclass[reqno]{amsart}

\usepackage{times}
\usepackage[T1]{fontenc}
\usepackage{mathrsfs}
\usepackage{latexsym}
\usepackage[dvips]{graphics}
\usepackage{epsfig}
\usepackage{enumitem}
\usepackage{verbatim}
\usepackage[hyphens]{url}
\usepackage{hyperref}
\usepackage[hyphenbreaks]{breakurl}
\usepackage{pdflscape}

\usepackage{amsmath,amsfonts,amsthm,amssymb,amscd}
\input amssym.def
\input amssym.tex
\usepackage{color}

\newcommand\be{\begin{equation}}
\newcommand\ee{\end{equation}}
\newcommand\bea{\begin{eqnarray}}
\newcommand\eea{\end{eqnarray}}
\newcommand\bi{\begin{itemize}}
\newcommand\ei{\end{itemize}}
\newcommand\ben{\begin{enumerate}}
\newcommand\een{\end{enumerate}}

\newtheorem{thm}{Theorem}[section]

\newtheorem{lem}[thm]{Lemma}
\newtheorem{prop}[thm]{Proposition}

\newtheorem{defi}[thm]{Definition}

\newtheorem{rek}[thm]{Remark}

\newtheorem{prob}[thm]{Problem}

\newcommand{\Z}{\ensuremath{\mathbb{Z}}}
\newcommand{\Q}{\mathbb{Q}}

\DeclareMathOperator{\Spor}{Spor}

\numberwithin{equation}{section}

\renewcommand{\l}{\ell}
\begin{document}

\title{On the Within-Perfect Numbers}




\author{Chung-Hang Kwan}
\email{\textcolor{blue}{\href{mailto: ucahckw@ucl.ac.uk}{ucahckw@ucl.ac.uk}}} \address{Department of Mathematics, University College London}


\author{Steven J. Miller}
\email{\textcolor{blue}{\href{mailto:sjm1@williams.edu}{sjm1@williams.edu}}, \textcolor{blue}{\href{mailto:Steven.Miller.MC.96@aya.yale.edu}{Steven.Miller.MC.96@aya.yale.edu}}} \address{Department of Mathematics and Statistics, Williams College, Williamstown, MA 01267}





\begin{abstract}
Motivated by the works of Erd\"os, Wirsing, Pomerance, Wolke and Harman on the sum-of-divisor function $\sigma(n)$, we study the distribution of a special class of natural numbers closely related to (multiply) perfect numbers which we term  `$(\l;k)$-within-perfect numbers', where $\ell >1$ is a real number and  $k: [1, \infty) \rightarrow (0, \infty)$ is an increasing and  unbounded function.

\end{abstract}

\maketitle


\section{Introduction}\label{introd}


 A natural number $n$ is said to be \emph{perfect} if $\sigma(n)=2n$,  \emph{$\l$-perfect} if $\sigma(n)=\l n$ (with $\ell> 1$ being rational), and  \emph{multiply perfect} if $n\ |\ \sigma(n)$, where  $\sigma(n)$ represents the sum of all positive divisors of $n$.
An outstanding conjecture, originating from the ancient Greeks (300 BC), asserts that there are infinitely many even perfect numbers but no odd perfect numbers (see Euclid  \cite[Proposition IX.36]{Eu}, Dickson \cite[Chapter I]{Di}, Guy \cite[Chapter B1]{Gu}). This conjecture is well-supported by probabilistic heuristics due to Pomerance  \cite[pp. 249, 258-259]{Pol}. For $\ell \in \{2, 3, \ldots, 11 \}$, there are known examples of  $\l$-perfect numbers (see  \cite[Chapter B1]{Gu}); however, for other values of $\ell$, the (non)existence of $\ell$-perfect numbers  remains entirely open. 

Starting in the mid-20th century, considerable interest emerged in understanding the statistical distribution of perfect numbers. These numbers are particularly \textit{rare}, as demonstrated by the works of  Erd\"os \cite{Er56},  Volkmann  \cite{Vo56}, Kanold \cite{Ka54, Ka57}, Hornfeck \cite{Ho55} and Hornfeck-Wirsing \cite{HW57}, culminating in the sharpest known \textit{upper bound} for the number of perfect numbers up to $x$  due to Wirsing \cite{Wi59}. The bound obtained in \cite{Wi59} is of the order $\exp(O(\frac{\log x}{ \log\log x})$ as $x\to\infty$, which possesses  the pleasant feature of \emph{uniform} applicability to
$\ell$-perfect numbers for any rational $\ell$. When restricting to the class of odd perfect numbers, there is the celebrated \textit{Dickson's finiteness theorem} \cite{Di13}  which asserts the following:  given a natural number $k$,  there are only finitely many odd perfect numbers with exactly $k$ distinct prime divisors. This was later refined by Pomerance \cite{Pom77} and  Heath-Brown \cite{HB94}. 


Subsequently, it evolved into an active research area to investigate special classes of natural numbers closely linked to perfect numbers, see \cite[Chapter B2-B3]{Gu}. For instance, Sierpi\'{n}ski \cite{Si65} introduced the notion of \emph{`pseudo-perfect numbers.'} In a companion article \cite{CC+20}, we studied a subclass of these numbers known as  \emph{`near-perfect numbers'}, proposed by Pollack-Shevelev \cite{PS12}. In \cite{CC+20}, we strengthened the results and analysis of \cite{PS12} by employing recursive partitions and sieve-theoretic techniques.  



In this article, we investigate another class of `approximate' perfect numbers with a somewhat different flavour, which we call the $(\l;k)$-\emph{within-perfect} numbers, where  $\l>1$ is a real number and $k=k(y)$ is a certain threshold function.  More precisely, a natural number $n$  is said to be $(\l;k)$-\emph{within-perfect} if the Diophantine inequality
\begin{align}\label{maDioineq}
    |\sigma(n)-\l n| \ < \ k(n)
\end{align}
holds. There are two distinct origins of these numbers. On one hand, Erd\"o{s} (\cite[pp. 46]{Gu}) and Makowski \cite{Ma79} were interested in the case when $\ell=2$ and $k$ is a constant. 
On the other hand,  the inequality (\ref{maDioineq}) arises naturally in the field of \textit{Diophantine approximations} for arithmetic functions.  Wolke \cite{Wo77} studied (\ref{maDioineq}) for any real $\ell>1$ and function $k(y)$ of the form $y^{c}$.  His result was improved by Harman \cite{Ha10} and very recently by J\"arviniemi \cite{Ja22+}. In \cite{Ja22+}, it was shown that for any real $\l > 1$ and any  $c\in(0.45,1)$, there exist \textit{infinitely many} $(\l;y^c)$-within-perfect numbers. The range of $c$ can be extended to $(0.39,1)$ under the \textit{Riemann Hypothesis} as indicated in \cite{Ha10}. The results of \cite{Ha10, Ja22+} rely on deep inputs from the distributions of primes in short intervals as well as that of the differences of consecutive primes, see \cite{BHP01, HB21, St22+, Ja22+}.  Interested readers are also referred to Alkan-Ford-Zaharescu \cite{AFZ09a, AFZ09b} for settings more general than (\ref{maDioineq}) and the related questions in  Diophantine approximations.

The main results of this article concern the class of threshold functions $k=k(y)$ which are complementary to those considered in \cite{Wo77, Ha10, Ja22+}.  Moreover, we are also interested in estimating the size of the set
\begin{align}
    W(\ell; k; x) \ := \ \left\{n\le x:  |\sigma(n)-\l n| \ < \ k(n)\right\}. 
\end{align}
Consequently,  our work employs a different set of techniques compared to the earlier mentioned works.


\begin{thm}\label{withinOrder2}
Let $c\in (0,1/3)$ be given. Suppose $k: [1, \infty) \rightarrow (0, \infty)$ is an increasing and  unbounded function satisfying $k(y)\le y^{c}$ for  $y\ge 1$.  Let $\Sigma$ be the set $\left\{ \sigma(m)/m:\, m\ge 1  \right\}$. Then

\begin{enumerate}[label=(\alph*)]
	
\item \label{mainbdd} If  \ $\l =a/b \in (\Q\cap (1,\infty))\setminus\Sigma$ with    $(a,b)=1$, then
\begin{equation}\label{ubbmain}
	\#W(\l;k;x)\ = \ O\left(ab^3 x^{2/3+c+o(1)}\right)
\end{equation}
for $1<\ell \le x^{c}$ and $x\ge 1$,  where the implicit constants are absolute. \newline

\item \label{mainasymp}If \ $\l = a/b\in\Sigma$ with  $(a,b)=1$, and there exists $\delta>0$ such that $k(y)\ge y^{\delta} $ for $y\ge 1$, then
\begin{equation}\label{eqn limexact}
  \lim_{x\to\infty} \,  \frac{\#W(\l;k;x)}{x/\log x}\  \ =  \ \   \sum_{\sigma(m)\, = \, \l m} \frac{1}{m}. 
\end{equation}


\end{enumerate}
\end{thm}

\begin{rek}
We also have an analogous result for $c=0$ and its proof is actually simpler, see Proposition \ref{allperf} and Remark \ref{constcon}. 
\end{rek}

\begin{rek}\label{reltoWir}
Firstly, the  infinite series of \eqref{eqn limexact} converges by Wirsing's Theorem (\cite{Wi59}), which will be applied in various ways throughout the course of proving Theorem \ref{withinOrder2}.

Secondly, a key strength of our theorem is that all dependencies in (\ref{ubbmain}) are made entirely explicit. Our bound remains non-trivial  even if  the numerator or denominator of $\ell$ grows with $x$ at a controlled rate,  though this necessitates appropriately shrinking the admissible range for $c$.  

Thirdly, while  it is possible that the dependence on $\ell$ (i.e.,  the factor $ab^3$) in (\ref{ubbmain}) can be improved, it does not appear that this dependence can be  removed.  This is  in contrast to Wirsing's Theorem.


\end{rek}


\subsection{Notations}

We use the following notations throughout this article:

\begin{itemize}
    \item $f(x) \asymp g(x)$ if there exist  constants $c_1, c_2>0$ such that $c_1g(x) < f(x) < c_2g(x)$ for sufficiently large $x$,

    \item $f(x) \sim g(x)$ if $\lim_{x \to \infty} f(x)/g(x) = 1$,

    \item $f(x) = O(g(x))$ or $f(x) \ll g(x)$ if there exists a  constant $C>0$ such that $f(x) < Cg(x)$ for sufficiently large $x$,

    \item $f(x) = o(g(x))$ if $\lim_{x \to \infty} f(x)/g(x) = 0$,

    \item subscripts indicate the  dependence of implied constants on other parameters,
    
    \item $p$ always denotes a prime number,
    
    \item  $\omega(n)$ denotes the number of distinct prime factors of $n$. 
\end{itemize}

\subsection{Acknowledgments}

This material is based upon work supported by the  EPSRC grant: EP/W009838/1,  the Williams SMALL REU Program under the  NSF grants DMS2341670, DMS1265673, DMS1561945 and DMS1347804, as well as the Swedish Research Council under grant no. 2021-06594 while the first author was in residence at Institut Mittag-Leffler in Djursholm, Sweden during the year 2024. Both of the authors would like to express their gratitude to the participants of the SMALL REU (2015-16) for their constant encouragement, and also for the generous hospitality in Williamstown and Djursholm.  It is a great pleasure to thank the reviewer for a careful reading of
the article and his/her valuable comments.

\section{Preliminary Discussions and Preparations}

\subsection{Distribution Function and Phase Transition}

We begin by briefly explaining why the sublinear regime for $k=k(y)$ is of the greatest interest in the study of within-perfectness. For this, we recall the definition of a distribution function.

\begin{defi}\label{rigdistr}
Let $-\infty\le a<b\le\infty$. 
\begin{enumerate}
    \item A function $F: (a,b)\rightarrow  \mathbb{R}$ is a distribution function if $F$ is increasing, right continuous, $F(a+)=0$, and $F(b-)=1$. 

    \item An arithmetic function $f:\mathbb{N}\rightarrow\mathbb{R}$ has a distribution function if there exists a distribution function $F$ such that
\begin{equation*}
    \lim_{x\to\infty}\frac{1}{x}\, \#\{n\le x: f(n)\le u\} 
 \ = \ F(u)
\end{equation*}
at all points of continuity of $F$.
\end{enumerate}
\end{defi}

A celebrated theorem of Davenport \cite{Da33}, \cite[Theorem 8.5]{Pol} asserts that $\sigma(n)/n$ possesses a continuous and strictly increasing distribution function on $[1,\infty)$. Let $D(\, \cdot\, )$ be the distribution function of $\sigma(n)/n$. We have $D(1)=0$ and $D(\infty)=1$, and for convenience, we extend the definition of 
 $D(\, \cdot\,)$  to $\mathbb{R}$ by setting $D(u)=0$ for $u<1$. More generally,  a necessary and sufficient criterion for the existence of distribution functions for additive functions was established by Erd\"{o}s-Wintner \cite{EW39},   see \cite[ Chapter I.5, III.4]{Te} for further details. 

The following result is an elementary consequence of Davenport's theorem. It describes the phase transition of the asymptotic density of the set of all  $(\l;k)$-within-perfect numbers, which we denote by $W(\ell; k)$.

\begin{prop}\label{phase}
Let $D(\cdot)$ be the distribution function of $\sigma(n)/n$. Then  
\begin{enumerate}[label=(\alph*)]

\item
\label{withinperf sub-linear}
If $k(n)=o(n)$, then $W(\l;k)$ has asymptotic density 0.

\item
\label{withinperf linear1}
If $k(n)\sim cn$ for some $c>0$, then $W(\l;k)$ has asymptotic density equal to  $D(\l+c)-D(\l-c)$.

\item
\label{withinperf linear2}
If $k(n)\asymp n$,  then $W(\l;k)$ has positive lower density and upper density strictly less than $1$.

\item
\label{withinperf superlinear}
If $n=o(k(n))$, then $W(\l;k)$ has asymptotic density 1.
\end{enumerate}
\end{prop}

\begin{proof}[Sketch]
The proof is elementary and we will only indicate the details for part \ref{withinperf superlinear}.  For any $j\in\mathbb{N}$, there exists $n_{j}\in\mathbb{N}$ such that $n/k(n)< 1/j$ for any $n\ge n_{j}$. We have
\begin{align}\label{ineqcom}
\hspace{5pt} \frac{1}{x}\, \#\{n\le x: \lvert \sigma(n)-\l n \rvert <jn\} \ 
\le \ \frac{n_{j}}{x} \ + \ \frac{1}{x} \, \#\{n\le x: \lvert \sigma(n)-\l n \rvert <k(n)\}
\end{align}
whenever $x\ge n_{j}$. If $j>\ell$, the left-hand side of (\ref{ineqcom}) converges to $D(\ell+j)$ as $x\to\infty$, and hence
\begin{equation*}
	\liminf_{x\to\infty} \, \frac{\#W(\ell;k;x)}{x} \ \ge \ D(\ell+j).
\end{equation*}
The desired result follows at once by taking  $j\to\infty$, using the fact that $D(\infty)=1$. 

\end{proof}

\subsection{On Congruences involving $\sigma(n)$} \label{pom}

Henceforth, our focus shifts to the sublinear thresholds, where we crucially make use of  the techniques developed by  Pomerance and his co-authors over the years.   In 1975, Pomerance \cite{Pom75} initiated the study of the equation 
\begin{align}\label{linshift}
	\sigma(n) \ = \ \l n \ +\ k
\end{align}
with $\l$, $k$ being integers and  $\ell>1$.
Central to \cite{Pom75} is the following important concept, which proves very useful in many  Erd\"os-style problems (see \cite{PP16}):

\begin{defi}[Regular/ Sporadic]\label{regspor}
	The solutions to the congruence \ $\sigma(n) \equiv k\  (\bmod\, n)$  of the form 
	\begin{equation}
		n=pm,  \hspace{10pt}  \text{ where } \hspace{10pt}  \ p \nmid m,\hspace{5pt}  m \mid \sigma(m), \hspace{5pt}  \sigma(m) \ = \  k,
	\end{equation}
	are called \emph{regular}. All other solutions are called  \emph{sporadic}.
\end{defi}


The main observation is that sporadic solutions occur much less frequently than regular solutions. In a series of works \cite{Pom75, Pom76a, Pom76b, PS12,  APP12, PP13}, this was quantified with various degrees of precision and uniformity. We summarize the progress made in this direction. 

Let $\Spor_{k}(x)$ be the set of sporadic solutions of $\sigma(n) \equiv k \ (\bmod\  n)$ up to $x$. In \cite{Pom75}, \cite{PS12},  \cite{APP12} and \cite{PP13}, it was shown, respectively,  that  the following bounds hold as $x\to \infty$:
\begin{itemize}
	\item $\#\Spor_{k}(x)= O_{\beta, k}(x\exp(-\beta\sqrt{\log x \log\log x}))$ for \emph{fixed} $k$ and \emph{fixed} $\beta<1/\sqrt{2}$, 
	
	\item $\#\Spor_{k}(x)= O(x^{2/3+o(1)})$ \emph{uniformly} in $\lvert k \vert < x^{2/3}$, 
	
	\item  $\#\Spor_{k}(x)=O(x^{1/2+o(1)})$ \emph{uniformly} in $\lvert k \vert < x^{1/4}$, and
	
	\item $\#\Spor_{k}(x)\cap \{\sigma(n) \ \text{is odd}\}=O(|k|x^{1/4+o(1)}) $ uniformly in $0<\lvert k \vert \le x^{1/4}$.
\end{itemize}
When $k=0$, we have a much stronger bound for (\ref{linshift}) due to   Wirsing (\cite{Wi59}).


Unfortunately, the aforementioned results are not sufficient for our study of within-perfect numbers. We must consider  a more general congruence
\begin{equation}\label{eqn Pomcong}
b\sigma(n) \ \equiv \ k \ (\bmod\, n)
\end{equation}
where $b$ and  $k$ are integers with $b\ge 1$. Accordingly, the definitions of regular and sporadic solutions should be  extended as follows:

\begin{defi}\label{extReg}
Suppose $b \mid k$. Then $n$ is said to be a \emph{regular solution} to \eqref{eqn Pomcong} if 
\begin{equation}
n\ = \ pm, \hspace{10pt}\text{ where } \hspace{5pt} p\nmid m, \hspace{5pt}  m \mid b\sigma(m),  \hspace{5pt} \sigma(m) \ =\ \frac{k}{b}.
\end{equation}
All other solutions are called  \emph{sporadic}. In the case when  $b\nmid k$,  all solutions to  \eqref{eqn Pomcong} are considered \emph{sporadic}.
\end{defi}

We record the following generalization of \cite[Lemma 8]{PS12}, which will be used in the proof of Theorem \ref{withinOrder2} and may also be of independent interest. For the convenience of the readers, we include a sketch of proof of this proposition here.

\begin{prop}\label{modified}
The number of sporadic solutions $n\le x$ to the congruence  $b\sigma(n)  \equiv  k  \, (\bmod\, n)$ is  $O(b^2x^{2/3+o(1)})$ for any $x\ge 1$ and   integer $k$ satisfying $|k| < bx^{2/3}$. The implicit constants are absolute. 
\end{prop}

	
Our intended applications take into account the uniformity of the range and the strength of the upper bound in  counting sporadic solutions. In light of Wirsing's Theorem  (\cite{Wi59}) and Remark \ref{reltoWir}, it is also desirable to maintain all implicit constants absolute but this can be somewhat subtle, see \cite[Remark 1]{APP12} and \cite{PP13}.   After a careful examination of existing strategies, the authors believe that  the approach of \cite{PS12} is the most suitable for attaining our desired generalization (e.g., $\ell$ can be rational), incorporating   all the favourable features mentioned above.  Their method  softly utilizes the unique factorization of a natural number into its square-free and square-full parts, the anatomy of unitary divisors \footnote{Note: $m$ is a unitary divisor of $n$ if $n$ has a decomposition of the form $n=mm'$, where $(m,m')=1$.}, and the following result from \cite[Theorem 1.3]{Pol11}: 
	\begin{equation}\label{polema}
		\sum_{x^{1/3}<m\le x^{2/3}}  \ \frac{(m,\sigma(m))}{m^2}  \ \le \   3x^{-1/3+o(1)}, 
	\end{equation}
	which turns out to be a nice application of  Wirsing's Theorem!
	

\begin{proof}[Proof of Proposition \ref{modified}]

We can certainly assume that $x\ge b$, otherwise the count is  trivially bounded by $b$, which is acceptable in view of the bound claimed in Proposition \ref{modified}. Let   $|k|<bx^{2/3}$. Suppose $n\le x$ is a sporadic solution to the congruence $b\sigma(n) \equiv k \,  (\bmod\, n)$. One can  simultaneously assume that $n>x^{2/3}$ and the square-full part of $n$ is  bounded by $x^{2/3}$. Indeed,  the contribution from the complement can easily seen to be $O(x^{2/3})$. We then consider the following two cases. 

\begin{enumerate}
\item Suppose $p:= P^{+}(n) >x^{1/3}$. By the assumption made, we have $n=pm$ with  $p \ \nmid \ m$ \,and\,  $m < x^{2/3}$. The congruence can be written as  $b\sigma(n) \ = \ qn+k$ for some integer $q\ge 0$. It follows that
\begin{align*}
b(p+1)\sigma(m) \ = \ b\sigma(n) \ = \ qpm+k,
\end{align*}
and
\begin{equation*}
p(b\sigma(m)-qm) \ = \ k-b\sigma(m).
\end{equation*}

If $k-b\sigma(m)=0$, then $n$ is a regular solution, which is a contradiction! So, $k-b\sigma(m) \neq 0$.  For each $m$, the number of such $p$  is $O(\log |k-b\sigma(m)|) \, = \, O(\log (bx^{2/3}\log\log x)) \, = \, O(\log bx)$ because of  $p\mid (k-b\sigma(m))$.  Therefore, the number of such $n$ is $O(x^{2/3}\log bx)= O(x^{2/3}\log x)$, which is acceptable. \newline

\item Suppose $p:= P^{+}(n) \le x^{1/3}$. Such $n$ must have a \emph{unitary divisor} $m$ in the interval $(x^{1/3},x^{2/3}]$, see \cite{PS12}.  We have $\sigma(n) \, \equiv \, 0 \ (\bmod\, \sigma(m))$ and $b\sigma(n) \, \equiv \, k \ (\bmod\, m)$. 
By the Chinese Remainder Theorem, we have $b\sigma(n) \equiv a_{n} \ (\bmod\  [m,\sigma(m)])$ for some unique $0\le a_{n}<[m,\sigma(m)]$. Given $m\in(x^{1/3},x^{2/3}]$, the number of possible values for  $b\sigma(n)$ is $\le 1+(2bx\log x)/ [m,\sigma(m)]$. 
Summing over $m\in(x^{1/3},x^{2/3}]$, the total  number of possible values of $b\sigma(n)$ is 
\begin{align*}
	\ &\le \ \sum_{x^{1/3}<m\le x^{2/3}} \ \left(\frac{2bx\log x}{[m,\sigma(m)]} \ + \ 1\right) \nonumber\\
	\ &\le \ x^{2/3} \ + \ (2bx\log x)(3x^{-1/3+o(1)}) \ \le  \ 7bx^{2/3+o(1)} 
\end{align*}
from (\ref{polema}). Moreover, the size of  $q=  (b \sigma(n)-k)/n$ is clearly
\begin{align*}
\hspace{25pt} \ \ll \ b\log\log x\ + \  \frac{|k|}{n} \ < \ b\log\log x+b \ \ll \ b\log\log x
\end{align*}
by the assumptions $n> x^{2/3}$ and $|k|<bx^{2/3}$. 
Since $b\sigma(n)=qn+k$, the number of possible values of $n$ is at most 
\begin{equation*}
    (7bx^{2/3+o(1)})(b\log\log x) \ \le \ 7b^2 x^{2/3+o(1)}. 
\end{equation*}
\end{enumerate}
The desired result follows by putting the conclusions of the two cases together.  
\end{proof}

\subsection{A simple application}


Let $k,a ,b$ be  integers such that $k\neq 0$,  $a>b\ge 1$, and  $(a,b)=1$. Denote by $S(a,b; k)$ the set of all solutions to the Diophantine equation $b\sigma(n)= an+k$ that generalizes (\ref{linshift}).  Let  $S(a,b;k; x):= S(a,b; k) \, \cap \, [1,x]$. The main result of  \cite{Pom75} is stated as follows. 
\begin{thm}\label{PoBd}
As $x\to\infty$, we have
\begin{equation}\label{eqn PoBd1}
\# S(a,1; k;x) \ \ll_{k, a} \ \frac{x}{\log x}.
\end{equation}
\end{thm}

Motivated by Pomerance's theorem, Davis-Klyve-Kraght \cite{DKK13} recently performed an extensive numerical investigation on the true size of $S(a,1; k;x)$. As a first application of Proposition \ref{modified}, we sharpen Pomerance's theorem and confirm some of the observations and speculations made in  \cite{DKK13}.  This will also serve as the base case for Theorem \ref{withinOrder2}, see Section \ref{lbmainth}.  


\begin{prop}\label{allperf}
\
\begin{enumerate}[label=(\alph*)]
\item
\label{almperf lperf}
If \ $k\ge 1$, $ab \mid k$,  and\,    $\sigma(k/a)=k/b$, then
\begin{equation}\label{eqn perfect order}
    \#S(a,b;k; x) \ \sim \ \frac{a}{k}\, \frac{x}{\log (ax/k)}
\end{equation}
as $x\to\infty$. 
\item
\label{almperf notlperf}
Otherwise, we have
\begin{equation}\label{eqn non perfect order}
    \#S(a,b;k; x) \ = \ O(b^2x^{2/3+o(1)})
\end{equation}
for any $x\ge 1$, where the implicit constants are absolute.  In particular, the bound \eqref{eqn non perfect order} is uniform in $a$. 
\end{enumerate}
\end{prop}

\begin{proof}[Proof]
Suppose $n\in S(a,b; k)$. Then $b\sigma(n)\equiv k\, (\bmod\, n)$. If $b\nmid k$, then all solutions are sporadic by Definition \ref{extReg}  and   \eqref{eqn non perfect order} follows at once from Proposition \ref{modified}. 

Suppose $b\mid k$ and $n$ is a regular solution. Then $k(1+p) =  b(1+p) \sigma(m) =  apm+k$. Hence,  $ b\sigma(m) = k = am$ and in particular, we have
\begin{align}\label{Pomconc}
    a\mid k \hspace{10pt} \text{ and }  \hspace{10pt} \sigma(k/a) \ = \ k/b \hspace{5pt}   (= \  \sigma(m)).  
\end{align}
In other words, if  (\ref{Pomconc}) is violated, then all solutions are  sporadic and  once again  \eqref{eqn non perfect order} holds.  This proves part \ref{almperf notlperf} of Proposition \ref{allperf}.

Suppose  $ab \mid k$ and $\sigma(k/a) \ = \ k/b$.  Then the  set $ \{n\in \mathbb{N}: n=p(k/a), \, p\nmid (k/a)\}$
consists of all regular solutions and is contained in $S(a,b; k)$.  Using Proposition \ref{modified}, the Prime Number Theorem, and the bound $\omega(n) = O(\log n)$, it follows that 
\begin{align*}
\#S(a,b; k;x) \ &=  \  \pi(ax/k) \ + \  O(\log |k|) \ + \  O(b^2x^{2/3+o(1)})  \nonumber\\
\ &\sim \  \frac{a}{k}\frac{x}{\log (ax/k)}
\end{align*}
as $x\to \infty$.  Hence, part \ref{almperf lperf} follows and this completes the proof of Proposition \ref{allperf}. 
\end{proof}

\section{Proof of Theorem \ref{withinOrder2}}

\subsection{Upper Bound}

We begin by proving the harder part of Theorem \ref{withinOrder2}, i.e., the upper bounds for $W(\ell; k;x)$. Fix $c\in(0,1/3)$. Let $k_{c}(y):= y^{c}$ and $k(y)\le k_{c}(y)$  for  $y\ge 1$.  Given  a function $f$, we write $\widetilde{W}(\l;f;x) := \{n\le x: |\sigma(n)-\l n|<f(x)\}$. The following inequality is apparent: 
\begin{equation}\label{eqn ubbineq}
\#W(\l;k;x) \ \le \  \#\widetilde{W}(\l;k;x) \ \le \  \# \widetilde{W}(\l; k_{c};x). 
\end{equation}


Let $\ell= a/b>1 $ be in the lowest term. For $n\in \widetilde{W}(\ell; k_{c}; x)$, we have  \, $b\sigma(n)-an=k$  \, for some integer $k$ such that $|k|<bx^{c}$.  In particular, we have
\begin{equation}\label{ourPomcong}
b\sigma(n) \, \equiv\,  k \ (\bmod\, n) \hspace{10pt}  \text{ for } \hspace{10pt} |k|<bx^{c}.
\end{equation}

By Proposition \ref{modified}, the number of $n\in \widetilde{W}(\ell; k_{c} ; x)$ which is a sporadic solution (see Definition \ref{extReg}) is  $O(bx^{c})\, \cdot\,  O(b^2x^{2/3+o(1)})=O(b^3x^{2/3+c+o(1)})$, which is acceptable  in view of Theorem \ref{withinOrder2}. On the other hand,  the number of regular solutions $n=pm\in \widetilde{W}(\ell; k_{c};x)$  with $p\le bx^{c}$ is $O(bx^{c+o(1)})$ by Wirsing's Theorem (\cite{Wi59} or see Section \ref{introd}). The contribution is clearly negligibly small. 

Suppose $p>bx^{c}$ and $b\sigma(m)=rm$ \ for some integer $r\ge a+1$. Then
\begin{align*}
\hspace{10pt} |b\sigma(n)-an|&=|b(1+p)\sigma(m)-apm|=|(1+p)rm-apm|=m|r+p(r-a)|\nonumber\\
&\ge m(r+p) > bx^{c}.
\end{align*}
This contradicts with (\ref{ourPomcong})!

Suppose $p>bx^{c}$ and $b\sigma(m)=rm$ \ for some integer $r\le a-1$. Then
\begin{itemize}
    \item If  $r+p(r-a)\ge 0$ (which implies $p<a$), then a contradiction with $p>bx^{c}$ arises  whenever $x^{c}>\ell$. 

    \item If $r+p(r-a)<0$, then $|b\sigma(n)-an|< bx^{c}\iff m[(a-r)p-r]<bx^{c}$. By Merten's Theorem, the  number of such $n$ is at most 
\begin{align*}
\sum_{2\le r \le a-1} \sum_{bx^{c}<p\le x} \frac{bx^{c}}{(a-r)p-r} \ & \ \le abx^{c} \sum_{bx^{c}<p\le x} \frac{1}{p-(a-1)} \nonumber\\
&\le 2abx^{c} \,\sum_{p\le x} \frac{1}{p} \ \ll \ abx^{c}\log\log x
\end{align*}
whenever $x^{c}\ge 2\ell$. The contribution is once again negligible. 

\end{itemize}

 Hence, we are left to consider the case when $r=a$, i.e.,  
 \begin{align}\label{siftremain}
     n \ = \ pm \ \in \  \widetilde{W}(\ell; k_{c};x) \hspace{10pt} \text{ such that}   \hspace{10pt}  p \ > \ bx^{c}  \hspace{5pt} \text{ and }  \hspace{5pt}  \sigma(m) \ = \ \ell m.
 \end{align}
If $\ell\not\in \Sigma$, there is clearly no such $n$ and thus
$\#W(\ell;k;x) =  O(ab^3x^{2/3+c+o(1)})$
by taking into account the paragraphs right above. This proves part \ref{mainbdd} of Theorem \ref{withinOrder2}. 

Suppose $\ell\in \Sigma$. Firstly, observe from partial summation and Wirsing's Theorem that
\begin{align}\label{parbdd}
	\sum_{\substack{\sigma(m) \, = \, \l m}} \, \frac{\log m}{m} \ = \ \int_{1}^{\infty} \frac{\log t}{t} \, dP_{\ell}(t) \ = \  \lim_{t\to \infty} \,  \frac{\log t}{t^{1-o(1)}}\, + \, \int_{1}^{\infty} \, \frac{\log t}{t^{2-o(1)}} \ dt\ \ll \  1,
\end{align}
where $P_{\ell}(t):=\#\{m\le t: \sigma(m) \, = \, \l m\}$. As a result, both of the series
\begin{equation}\label{conv}
	\sum_{\sigma(m) \, = \, \l m}  \, \frac{\log m}{m} \hspace{10pt} \text{ and  } \hspace{10pt}  \sum_{\sigma(m) \, = \, \l m}\,  \frac{1}{m}
\end{equation}
are readily seen to be convergent.  Notice that the bound (\ref{parbdd}) is uniform in\, $\ell$. 

Secondly, we have $m<x^{1-c}$\, since \ $x\ge n=pm> x^{c}m$. Then 
\begin{equation*}
    0\ < \ \frac{\log m}{\log x} \ \le \  1-c \ < \ 1,
\end{equation*}
and 
\begin{equation}\label{GPexp}
    \bigg(1-\frac{\log m}{\log x}\bigg)^{-1} \ = \ 1+O_{c}\bigg(\frac{\log m}{\log x}\bigg).
\end{equation}
Let $\beta>1$ be given. The Prime Number Theorem implies  the existence of a constant $X_{0}=X_{0}(\beta)>0$ such that  $\pi(x)<\beta x/\log x$ whenever $x\ge X_{0}$. For $x\ge X_{0}^{1/c}$, the number of $n$ satisfying (\ref{siftremain}) is at most
\begin{align*}
    \sum_{\substack{\sigma(m)=\ell m\\ m\le x^{1-c}}} \pi (x/m) \ &< \ \beta\ \sum_{\substack{\sigma(m)=\ell m\\ m\le x^{1-c}}}\frac{x/m}{\log (x/m)}\nonumber\\
    \ &< \  \frac{\beta x}{\log x} \sum_{\sigma(m)=\l m} \frac{1}{m} \, + \,  O_{c}\bigg(\frac{\beta x}{(\log x)^2}\sum_{\substack{\sigma(m)=\l m}}\frac{\log m}{m}\bigg)\nonumber\\
    \ &<  \ \frac{\beta x}{\log x} \sum_{\sigma(m)=\l m} \frac{1}{m} \, + \,  O_{c}\bigg(\frac{\beta x}{(\log x)^2}\bigg)
\end{align*}
with the help of (\ref{GPexp}) and the convergence of the series in  (\ref{conv}).  Therefore, we have
\begin{equation}\label{eqn upperlim}
    \limsup_{x\to\infty} \, \frac{\#W(\l;k;x)}{x/\log x} \ \le\  \beta \sum_{\sigma(m)=\l m} \frac{1}{m}
\end{equation}
for any $\beta >1$. Let $\beta \to 1+$ in \eqref{eqn upperlim}, it follows that
\begin{equation*}
    \limsup_{x\to\infty}\frac{\# W(\l;k;x)}{x/\log x} \ \le \ \sum_{\sigma(m)=\l m} \frac{1}{m}.
\end{equation*}
This proves the upper bound of Theorem  \ref{withinOrder2}.\ref{mainasymp}.


\subsection{Lower Bound}\label{lbmainth}

 The proof of the lower bound for  Theorem \ref{withinOrder2}.\ref{mainasymp} is  relatively straight-forward.  Suppose $\ell\in \Sigma$. Based on the experience of Pomerance et. al. (see  Section \ref{pom}), a lower bound for  $\# W(\ell; k; x)$ can be obtained by estimating the size of the set 
\begin{align}
\mathcal{L}_{\ell}(\delta) \ := \ 	\left\{ n\le x: \, n \, = \, pm, \   p\nmid m, \ \sigma(m) = \ell m, \ \ell m \, < \, (pm)^{\delta}  \right\}
\end{align}
provided that  $k(y)\ge y^{\delta}$ for any $y\ge 1$, where $\delta\in (0,1)$.  We have
\begin{align*}
	\# \mathcal{L}_{\ell}(\delta) \ = \  \sum_{\substack{\sigma(m)=\ell m \\ m\le x^{\delta}/\ell}} \ \sum_{\substack{(\ell m)^{1/\delta}/m < p \le x/m \\ p\, \nmid\,  m}} \, 1. 
\end{align*}

Using the bound $\omega(m)= O(\log m)$, Wirsing's theorem and partial summation, it follows that
\begin{align*}
		\# \mathcal{L}_{\ell}(\delta) \ = \  \sum_{\substack{\sigma(m)=\ell m \\ m\le x^{\delta}/\ell}} \ \sum_{\substack{(\ell m)^{1/\delta}/m < p \le x/m }} \, 1 \  \ + \ \   O_{\delta}(x^{o(1)}),
\end{align*}
and 
\begin{align*}
	\sum_{\substack{\sigma(m)=\ell m \\ m\le x^{\delta}/\ell}} \ \sum_{\substack{ p \le  (\ell m)^{1/\delta}/m  }} \, 1  \ \ll \   \ell^{1/\delta}	\sum_{\substack{\sigma(m)=\ell m \\ m\le x^{\delta}}}  \ m^{1/\delta-1} \ \ll \  \ell^{1/\delta}	 x^{1-\delta+o(1)}. 
\end{align*}
Hence, 
\begin{align}
	\# \mathcal{L}_{\ell}(\delta) \ = \  \sum_{\substack{\sigma(m)=\ell m \\ m\le x^{\delta}/\ell}} \ \pi(x/m) \  \ + \ \   O_{\delta}\left(\ell^{1/\delta}	 x^{1-\delta+o(1)}\right). 
\end{align}

Let $\alpha<1$. By the Prime Number Theorem, there exists $x_{0}=x_{0}(\alpha)>0$ such that  $\pi(x)>\alpha x/\log x$ whenever $x\ge x_{0}$.  Thus, if $x>(x_{0})^{1/(1-\delta)}$, then
	\begin{align*}
		\# \mathcal{L}_{\ell}(\delta) \ &>  \  \frac{\alpha x}{\log x} \sum_{\substack{\sigma(m)=\ell m \\ m\le x^{\delta}/\ell}} \  \frac{1}{m} \  \ + \ \   O_{\delta}\left(\ell^{1/\delta}	 x^{1-\delta+o(1)}\right) \nonumber\\
		\ &= \  \frac{\alpha x}{\log x} \left( \sum_{\substack{\sigma(m)=\ell m }} \  \frac{1}{m}  \, + \, O\left( (\ell /x ^{\delta})^{1+o(1)}\right)\right) \  +  \   O_{\delta}\left(\ell^{1/\delta}	 x^{1-\delta+o(1)}\right). 
	\end{align*}
	From this, we may deduce that
	\begin{align}
	\liminf_{x\to \infty} \ 	\frac{\# W(\ell; k; x)}{x/\log x} \ > \  \alpha \sum_{\substack{\sigma(m)=\ell m }} \  \frac{1}{m}.  
	\end{align}
	Since this holds for any $\alpha < 1$, the lower bound for  Theorem \ref{withinOrder2}.\ref{mainasymp} follows.

\begin{rek}\label{constcon}
	Suppose $\ell \in \Sigma$.   Proposition \ref{allperf} implies that  when $k\equiv k_{0} \ge 1$ is a \emph{constant} function,  we have
	\begin{align}\label{consasymp}
		\frac{\# W(\ell; k_{0}; x)}{x/\log x} \ &= \ \sum_{|k|< bk_{0}} \, \frac{\#S(a, b;k;x)}{x/\log x} \ \sim \ \  \sum_{\substack{0< m< k_{0}/\ell \\ b\mid m \\ \sigma(m)=\ell m}} \  \frac{1}{m}
	\end{align}
	as $x\to \infty$. The rightmost quantity of (\ref{consasymp}) converges to $\sum\limits_{\substack{\sigma(m)=\l m \\ b\mid m}} \, 1/m$ as $k_{0} \to \infty$.  In particular,  when $\ell \in \Z $ and $k$  is an  increasing unbounded function, one readily observes that
	\begin{align}\label{onlyintasmp}
		\frac{\# W(\ell; k; x)}{x/\log x}  \ \sim \  \lim_{k_{0}\to \infty} \ 	\frac{\# W(\ell; k_{0}; x)}{x/\log x}
	\end{align}
	as $x\to \infty$. However, the asymptotic (\ref{onlyintasmp}) is not necessarily true when $\ell \not \in \Z$ because of the restriction $b\mid m $ present in (\ref{consasymp})!

\end{rek}

\section{Concluding Discussions, Numerics,  \& Further Directions}

Building upon the method of  \cite{APP12},  Pollack-Pomerance-Thompson \cite{PPT18} recently proved a variant of the main result of \cite{APP12}, albeit with a weaker error term and uniformity. Specifically, they proved that  if $\ell \in \Z$ is \emph{kept fixed}, the number of sporadic solutions to the \emph{equation} $\sigma(n) = \ell n+k$ up to $x$ is  $O(x^{3/5+o_{\ell}(1)})$ as $x\to \infty$ and for any integer $k$. To the best of the authors' knowledge,  there seems to be a number of subtleties in generalizing the method of \cite{PPT18} to the equation $b\sigma(n)= an+k$.  Additionally, as noted  in \cite{APP12}, it appears that obtaining an estimate that is fully uniform in all of  $a, b, k$  would require considerable effort. 

If  $\ell$ is restricted to be an \emph{integer}  and  is kept \emph{fixed}, the same argument from  Theorem \ref{withinOrder2} along with the main theorem of \cite{PPT18}, should yield the slightly improved admissible range  $c\in (0, 2/5)$. However,  the barrier of our method seems to be $c\in (0,1/2)$, see   \cite[Conjecture 4.3]{PPT18}.  When $\ell$ is \emph{not} an integer, it is unclear what the barrier should be and likely somewhat smaller than $(0,1/2)$.  




Theorem  \ref{withinOrder2} leads to  several interesting consequences which are stated as follows. As usual,  $k:[1, \infty) \rightarrow (0,\infty)$ is an increasing function such that
\begin{align}\label{bddcond}
\hspace{20pt} 	y^{\delta} \ \le \ k(y) \ \le \ y^{c} \hspace{20pt} \text{ for } \hspace{20pt} y \ \ge \ 1. 
\end{align}
 Firstly, it is natural to consider the following quantity
\begin{equation}
	\mathcal{D}_{c}(\ell) \ := \ \lim_{x\to\infty} \ \frac{\#W(\ell;y^{c};x)}{x/\log x} \ =: \   \lim_{x\to\infty} \  \mathcal{D}_{c}(\ell;x)
\end{equation}
for  $\ell \in [1,\infty)$ and $c\in (0,1)$. In view of Proposition \ref{phase}, this new `distribution function'  is arguably well-suited to study the within-perfect numbers with respect to the sublinear threshold.  However, this distribution function behaves quite differently from the ones described in Definition \ref{rigdistr}.


\begin{prop}\label{discont}
	The function  $\ell \mapsto \mathcal{D}_{c}(\ell )$ is discontinuous on a dense subset of\,  $[1,\infty)$, for any $c\in(0,1/3)$. 
\end{prop}

\begin{proof}
	It follows from a theorem of Anderson (see \cite[pp. 270]{Pol}) that $(\Q\cap[1,\infty))\setminus \Sigma$ is dense in $[1,\infty)$. Observe that $\mathcal{D}_{c}$ takes the value $0$ on $(\Q\cap[1,\infty))\setminus \Sigma$ but it takes positive values on $\Sigma$ by Theorem \ref{withinOrder2}. So, $\mathcal{D}_{c}$ is discontinuous on $\Sigma$. It is a well-known theorem that $\Sigma$ is again dense in $[1,\infty)$ (see \cite[pp. 275]{Pol}). This completes the proof. 
\end{proof}

Secondly, a real number  $\ell>1$ is said to be \emph{$\Sigma$-approximable}  if there exists  a function $f(x)\to \infty$  and a sequence of positive integers $(m_{i})_{i\ge 1}$ such that $	|\ell -  \sigma(m_{i})/m_{i}| < 1/f(m_{i})$ for any $i\ge 1$.  It is clear that

\begin{prop}
If $\ell>1$ is  $\Sigma$-approximable, then	$\#W(\ell; k; x) \gg x/\log x$  on an unbounded set of $x$.  
\end{prop}

 In fact, for any function $f(x)\to \infty$, there are \emph{irrational} numbers $\ell>1$ that are $\Sigma$-approximable by $f$.  This follows from the standard nested interval argument and the theorems of Anderson used in the proof of Proposition \ref{discont}.

We conclude this article with some numerics and  open problems for further investigation. 

 We calculate the quotient of $D_{c}(2;x)$ for various  values of $c\in (0,1)$ and  at $x = 1,000,000$, $x = 10,000,000$, and $x = 20,000,000$. Note: $\sum\limits_{\sigma(m)=2m} \frac{1}{m}\approx 0.2045$.

\begin{table}[h]
    \centering
    \begin{tabular}{c|c|c|c}
       $k(y)$  & $x = 1,000,000$ & $x = 10,000,000$ & $x = 20,000,000$ \\ \hline
       $y^{0.9}$ & 3.661860 & 3.305180 & 3.196040 \\
       $y^{0.8}$ & 1.141480 & 0.945623 & 0.908751 \\
       $y^{0.7}$ & 0.494278 & 0.435395 & 0.426470 \\
       $y^{0.6}$ & 0.311567 & 0.274586 & 0.267904 \\
       $y^{0.5}$ & 0.276559 & 0.259482 & 0.255962 \\
       $y^{0.4}$ & 0.264968 & 0.252956 & 0.250063 \\
       $y^{0.3}$ & 0.225980 & 0.247837 & 0.247299 \\
       $y^{0.2}$ & 0.151238 & 0.195911 & 0.197430 \\

    \end{tabular}
    \vspace{3mm}
    \caption{$\mathcal{D}_{c}(2;x)$ for various values of $x$ and $c$.}
    \label{tab:data_table}
\end{table}

It is natural to ask the following. 

\begin{prob}
	When $c\in (1/2,1)$, does  the correct order of magnitude for $\#W(\ell ;k;x)$, with $\ell\in \Sigma$ and $k$ satisfying (\ref{bddcond}),   continue to be $x/\log x$ as $x\to \infty$?

\end{prob}

In between the sublinear and linear regime, e.g., $k(y)=y/\log y$, Proposition \ref{phase} gives no conclusion. Consider the plot of $x\mapsto \#W(2; k; x)/(x/\log x)$ for such   $k$ from $x=2$ to $x=10,000$: 

\begin{figure}[h]
	\centering
	\includegraphics{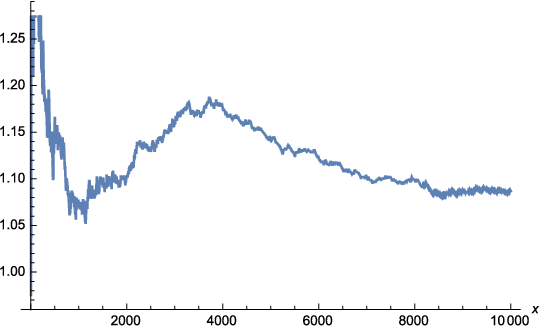}
	\caption{This plot shows the quantity $\#W(2;k;x)/(x/\log x)$
		with $k(y) = y/\log y$ for $x = 2$ to $10,000$.}
	\label{fig: WithinPerfxlogx}
\end{figure}

\begin{prob}
	What is the order of magnitude of $\#W(\ell;k;x)$ if the function $k$ satisfies $y^{c}=o(k(y))$ \emph{for any} $c\in(0,1)$?
\end{prob}

\begin{prob}
What is the order of magnitude for $\#W(\l;k;x)$ for \emph{irrational} $\l$? 
\end{prob}

\begin{prob}
Determine the  set of points of continuity for the distribution function $\ell \mapsto \mathcal{D}_{c}(\ell)$. 
\end{prob}



\bigskip

\end{document}